\documentclass[a4paper, 12pt, reqno]{amsart} 
\usepackage{transparent} 
\usepackage{enumitem}
\usepackage[colorinlistoftodos]{todonotes} 
\usepackage{amsmath, amsthm,bbm, amssymb, amsfonts, fullpage,hyperref} 
\usepackage[maxbibnames=99,backend=biber,bibencoding=utf8,style=numeric,giveninits=true,url=false,doi=false,isbn=false]{biblatex} 
\usepackage{mathtools} 
\usepackage{mathrsfs}
\mathtoolsset{showonlyrefs} 
\usepackage{graphicx, subfigure, color,calc} 
\numberwithin{equation}{section} 

\date{\today}

\theoremstyle{plain} 
\newtheorem{thm}{Theorem}[section] 
 
\newtheorem{cor}[thm]{Corollary} 
\newtheorem{lemma}[thm]{Lemma} 
\newtheorem{conjecture}[thm]{Conjecture} 
 
\newtheorem{prop}[thm]{Proposition} 
\theoremstyle{remark}

\theoremstyle{definition}
\newtheorem{defn}[thm]{Definition}
\newtheorem{obs}[thm]{Observation}

\newtheorem{conj}[thm]{Conjecture}

\newcommand{\lb}{\left(} 
\newcommand{\rb}{\right)}

\newcommand{\lvv}{\left\lvert} 
\newcommand{\rvv}{\right\rvert}

\renewcommand{\epsilon}{\varepsilon} 
\newcommand{\e}{\epsilon}

\renewcommand{\hat}{\overline}

\newcommand{\tmr}{\operatorname{tmr}} 
\newcommand{\rank}{\operatorname{rank}}
\newcommand{\inv}{\operatorname{inv}}

\title{On the failure of Additivity of inversion numbers for dijoins} 

\author{Natalie Behague} 
\author{Jack Cromwell} 
\email{natalie.behague@dcu.ie, jack.cromwell3@mail.dcu.ie}
\address{School of Mathematical Sciences, Dublin City University, Dublin, Ireland}

\begin{document}

 \begin{abstract}
     Let $D$ be an oriented graph. For $X \subseteq V(D)$, the inversion of $X$ in $D$ is the oriented graph that is obtained via changing the direction of all edges which have both endpoints in $X$. The inversion number of $D$, $\operatorname{inv}(D)$, is the minimum number of inversions necessary in order to remove all cycles from $D$. A dijoin $D_1 \to D_2$ is the oriented graph obtained by taking disjoint copies of $D_1,D_2$ and adding edges from all $v \in V(D_1)$ to each $u \in V(D_2)$.
     
     It was conjectured by  Bang-Jensen, da Silva and Havet that $\operatorname{inv}(D_1 \to D_2) = \operatorname{inv}(D_1) + \operatorname{inv}(D_2)$ for all $D_1, D_2$. This was disproved by simultaneous constructions of various $D_1$, $D_2$ with $\operatorname{inv}(D_1 \to D_2) = \operatorname{inv}(D_1)+\operatorname{inv}(D_2)-1$. However, no known constructions achieved more than this gain of 1. Moreover, all of the constructions exploit a gap of 1 between $\operatorname{inv}(D)$ and the minimum rank of a family of related matrices, a gap that can never be more than 1,  suggesting that perhaps something close to the original conjecture might be true.

     We prove that in fact additivity fails badly. In particular, for all $\varepsilon>0$ we show that there exist tournaments $D_1$, $D_2$ with arbitrarily large inversion numbers such that $\operatorname{inv}(D_1 \to D_2) \leq \left(\frac{1}{2} +\varepsilon\right)(\operatorname{inv}(D_1)+\operatorname{inv}(D_2))$.
 \end{abstract}

 \maketitle

 \section{Introduction and Statement of Results}

 Let $D$ be an oriented graph and $X \subseteq V(D)$. The \emph{inversion} of $X$ in $D$ is the oriented graph obtained by flipping the direction of every arc in $D$ that has both endpoints in $X$. For a family of sets $X_1,\dots,X_p$ the inversion of $X_1,\dots,X_p$ is obtained by inverting each set successively, where we note that the order of the inversion doesn't affect the final graph. This family is a \emph{decycling family} of $D$ if these inversions on $D$ produce an acyclic tournament. The \emph{inversion number} $\inv(D)$, is defined to be the smallest integer $p$ for which a $p$-decycling family of $D$ exists. It was first introduced by Belkhechine \cite{Bel}, with early results obtained by Belkhechine, Bouaziz, Boudabbous and Pouzet \cite{Belk}, largely concerning $\inv(n)$, the maximum inversion number of a tournament on $n$ vertices.

For oriented graphs $D_1, D_2$, the dijoin $D_1 \to D_2$ is the oriented graph formed by taking the disjoint union of $D_1$ and $D_2$, and adding all edges $uv$ where $u \in V(D_1)$ and $v \in V(D_2)$. The inversion number of dijoins was first examined by Bang-Jensen, da Silva and Havet \cite{Bang}. It is easy to see that $ \inv(D_1 \to D_2) \leq \inv(D_1) + \inv(D_2)$, by inverting first $D_1$ and then inverting $D_2$. Bang-Jensen, da Silva and Havet \cite{Bang} conjectured that the inversion number is additive, in the sense that
\begin{equation}\label{eq:inv_additivity_conj}
    \inv(D_1 \to D_2) = \inv(D_1) + \inv(D_2),
\end{equation}
and proved this conjecture in the special case of $\inv(D_1) = \inv(D_2) = 1$.

However, this conjecture was disproved by two simultaneous papers \cite{Alon} and \cite{Problems}, which provide examples with
\begin{equation*}
    \inv(D_1 \to D_2) = \inv(D_1) + \inv(D_2) -1.
\end{equation*}
These counterexamples were constructed very carefully, and in such a way that the arcs from $V(D_1)$ to $V(D_2)$ always end up all going in the original direction. In the opposite direction, Behague, Johnston, Morrison and Ogden \cite{Natnote} showed that the bound $\inv(D_1\to D_2)> \inv(D_1)$ holds whenever $\inv(D_1) = \inv(D_2)$. Since the trivial lower bound
\begin{equation*}
    \inv(D_1 \to D_2) \geq \max( \inv(D_1),\inv(D_2))
\end{equation*}
always holds, this leaves open the question of how far from additive the inversion number of a dijoin can be.

It was observed in \cite{Natnote} that the known counterexamples to \eqref{eq:inv_additivity_conj} in \cite{Alon,Problems} are closely tied to a linear algebraic relaxation of the inversion number, tournament minimum rank $\tmr(D)$, which is defined as the minimum rank of a symmetric matrix $M \in \mathbb{F}^2$ which decycles a graph $D$ by flipping all arcs $uv$ for which $M_{uv} = 1$ (see Section \ref{sec:tmr}). These two properties are very closely related: it is shown in \cite{Natnote} that
\begin{equation}
    \tmr(D) \leq \inv(D) \leq \tmr(D) + 1.  \label{rel}
\end{equation}
Further, the dijoins constructed as counterexamples to \eqref{eq:inv_additivity_conj}  preserve additivity of tournament minimum rank, and the saving of one inversion relies entirely on the gap between the two properties. In light of this, it was asked in \cite[Question 17]{Natnote} (see also \cite[Problem 30.2]{BCC}) whether additivity holds for tournament minimum rank, that is, whether
\begin{equation}
    \tmr(D_1 \to D_2) = \tmr(D_1) + \tmr(D_2) \text{ for all }D_1,D_2. \label{eq:tmr_adds}
\end{equation}
 This question keeps alive the spirit of the dijoin conjecture while also ruling out all known counterexamples. It is also the barrier to any substantial failure of additivity for the inversion number since $\inv(D)$ and $\tmr(D)$ differ by at most one so a large drop in $\inv(D_1 \to D_2)$ requires a large drop in $\tmr(D_1 \to D_2)$.
 
Equation \eqref{eq:tmr_adds} was confirmed to hold in the cases $\tmr(D_i) = 1$ and  $\tmr(D_i) = 2$  for some $i \in \{1,2\}$ by Wang, Yang and Lu \cite{Wang} and Behague and Gaudart-Wifling \cite{natcase} respectively. (Note that Wang, Yang and Lu \cite{Wang}  use different notation, but the implication is the same.) 
Behague and Gaudart-Wifling \cite[Conjecture 6.1]{natcase} conjectured that \eqref{eq:tmr_adds} always holds.

Perhaps surprisingly, we show that not only is the tournament minimum rank not additive on dijoins, it can, in fact, be very far from additive. This in turn implies that the inversion number can be very far from additive. We prove the following two related theorems, which show how badly additivity can fail.
\begin{thm} \label{T1}
    For all $\e>0$ there exist tournaments $D_1, D_2$ with $\inv(D_1),\inv(D_2)$ arbitrarily large such that
    \begin{equation*}
        \inv(D_1 \to D_2) \leq \lb\frac{1}{2} +\e\rb(\inv(D_1)+\inv(D_2)).
    \end{equation*}
\end{thm}
\bigskip 
\begin{thm} \label{T2}
    For all $n$ sufficiently large there exist tournaments $D_1$, $D_2$ on $n$ vertices with $\tmr(D_1),\tmr(D_2) = n - O(\sqrt{n\log{n}})$ such that 
    \begin{equation*}
       \tmr(D_1 \to D_2) \le n = \lb\frac{1}{2} +o(1)\rb(\tmr(D_1)+\tmr(D_2)).
    \end{equation*}
\end{thm}

\bigskip

After finding an example of $D_1$ and $D_2$ for which $\inv(D_1\rightarrow D_2) < \inv(D_1) + \inv(D_2)$, Alon, Powierski, Savery, Scott, and Wilmer \cite{Alon} conjectured that examples should exist for almost all possible values of $\inv(D_1)$ and $\inv(D_2)$ (except when both are $\le 2$).
\begin{conjecture}[\protect{\cite[Conjecture 8.9]{Alon}}]
    For all $\ell ,r\in \mathbb{N}$ with $\ell \geq 3$ or $r\geq 3$ there exist oriented graphs $D_1$ and $D_2$ with $\inv(D_1) =\ell$ and $\inv(D_2) =r$, but $\inv(D_1\rightarrow D_2)< \inv(D_1) + \inv(D_2)$.
\end{conjecture}
Wang, Yang and Lu~\cite{Wang} disproved this conjecture by showing that it is false for $\ell=1$ and $r$ even (or vice versa). However, as a corollary of Theorem~\ref{T1}, we recover a version of the conjecture for sufficiently large $\ell$ and $r$.
\begin{thm}\label{thm:all_ell_r}
    There exists $k \in \mathbb{N}$ such that for all $\ell ,r\in \mathbb{N}$ with $\ell,r \geq k$, there exist tournaments $D_1$ and $D_2$ with $\inv(D_1) =\ell$ and $\inv(D_2) =r$, but $\inv(D_1\rightarrow D_2)< \inv(D_1) + \inv(D_2)$. 
\end{thm}

In the following section we will provide key definitions, some prior results relating symmetric matrices to tournaments, and some counting arguments from the literature. The theorems above are then proved in Section~\ref{sec:main}. We finish with a discussion of open problems and further work in Section~\ref{sec:conclusion}.

\section{Preliminaries}

\subsection{Notation and conventions}
  Throughout, we take an \emph{oriented graph} to be a vertex-labelled directed graph with no loops and at most one arc between any two vertices. A \emph{tournament} is an oriented graph with exactly one arc between any two vertices. In the rest of the paper we will only deal with tournaments. A tournament is \emph{acyclic} if and only if it is \emph{transitive}, meaning its arcs all point from the smaller to the larger element of some linear order on its vertices. All matrices have entries in $\mathbb{F}_2$ unless stated otherwise. We write $S_n$ for the set of invertible symmetric $n \times n$ matrices and $L$ for the $n \times n$ lower triangular matrix with all entries on and below the diagonal equal to 1, noting that $L$ is invertible.

\subsection{Tournament Minimum Rank}\label{sec:tmr}
We first formalize the notion of tournament minimum rank, as introduced in \cite{Natnote}.

\begin{defn}
    Let $D$ be a tournament with $V(D)=(v_1,v_2,\dots,v_n)$. We say that an $n\times n$ symmetric matrix $M$ with entries in $\mathbb{F}_2$ is a \emph{decycling matrix} for $D$, if reversing the orientation of all arcs $v_iv_j$ where $m_{ij}=1$, results in a tournament with no directed cycles.     
\end{defn}
\begin{defn}
Let $D$ be a tournament. The \emph{tournament minimum rank} of $D$, denoted $\tmr(D)$, is the smallest rank of a decycling matrix for $D$.
\end{defn} 

We note that the diagonal entries of $M$ do not affect which arcs are reversed, however they do affect the rank of $M$. This freedom is what allows $\tmr(D)$ to differ from $\inv(D)$. 

\begin{obs} \textbf{(Monotonicity)}
    If $D'$ is an induced subtournament of $D$ then $\tmr(D') \leq \tmr(D)$. This implies
    \begin{equation*}
        \tmr(D_1 \to D_2) \ge \max\{\tmr(D_1),\tmr(D_2)\} \ge \frac{1}{2} ( \tmr(D_1) + \tmr(D_2)).
    \end{equation*}
\end{obs}

To see this, let $M$ be a decycling matrix for $D$ of rank equal to $\tmr(D)$, and let $M'$ be its principal submatrix indexed by $V(D')$. Reversing the arcs of $D'$ according to $M'$ gives a subtournament of an acyclic tournament and hence is acyclic, and $\rank(M') \leq \rank(M)$. The same argument applies to the inversion number so the bounds in Theorems \ref{T1} and \ref{T2} are best possible up to the $o(1)$ factor.

The following construction allows us to pass from symmetric matrices to tournaments.

\begin{defn}
    For a symmetric $n \times n$ matrix $G$, let $T(G)$ be the tournament on vertex set $[n] = \{1,\dots,n\}$ in which, for $i < j$, the arc is $i \to j$ if $G_{ij} = 0 $ and $j \to i$ if $G_{ij}=1$.
\end{defn}

\begin{obs} \label{G1}
    Let  $G$ be a symmetric $n \times n$ matrix.
    \begin{enumerate}
        \item $G$ is a decycling matrix for $T(G)$, as reversing the arcs of $T(G)$ according to $G$ yields the transitive tournament $1\to2\to\dots\to n$.
        \item $T(G)$ depends only on the entries of $G$ above the diagonal (i.e.~does not depend on the diagonal entries) and $G \to T(G)$ induces a bijection between symmetric matrices with zeros on the diagonal and tournaments on $[n]$.
        \item For every tournament $D$ on $[n]$ there are exactly $2^n$ symmetric matrices $G$ with $T(G) = D$, one for each choice of diagonal.  \label{item:2^n diagonal choices}
    \end{enumerate}
\end{obs}

\bigskip

\subsection{Counting Arguments}
We will need to bound the number of symmetric matrices of given ranks. The following count is due to MacWilliams \cite{MacW}. 

\begin{lemma}[{\cite[Theorem 2]{MacW}}] \label{lem:Mc1}
    Let $N(n,k)$ denote the number of $n \times n$  symmetric matrices with rank $k$ and with entries in a finite field $\mathbb{F}_q$. For $0 \leq k \le n$,
    \begin{equation*}
        N(n,k) = \prod_{i=1}^{\left \lfloor{k/2}\right \rfloor } \frac{q^{2i}}{q^{2i}-1} \times \prod_{i=0}^{k-1} (q^{n-i}-1).
    \end{equation*}
\end{lemma}

Applying the above result with $k =n$ and $q=2$ gives the following easy corollary.
 \begin{cor}  \label{lem:Mc2}
     Let $S_n$ denote the set of invertible $n\times n$ symmetric matrices over $\mathbb{F}_2$. We have
     \begin{equation*}
         |S_n| = 2^{n(n+1)/2} \prod_{1 \leq i \le n, i \text{ odd  }} \lb1-\frac{1}{2^i}\rb
     \end{equation*}
 \end{cor}

In addition, we can apply Lemma~\ref{lem:Mc1} to obtain a bound on the number of tournaments with small tournament minimum rank.
 \begin{cor} \label{lem:count}
     Let $M(n,k)$ denote the number of tournaments on $n$ vertices with ${\tmr \leq k}$. Then 
     \begin{equation*}
         M(n,k) \leq n! \sum_{j=0}^k N(n,j) = n!  \sum_{j=0}^k \prod_{i=1}^{\left \lfloor{j/2}\right \rfloor } \frac{2^{2i}}{2^{2i}-1} \times \prod_{i=0}^{j-1} (2^{n-i}-1)
     \end{equation*}
 \end{cor}

 \begin{proof}
     If $\tmr(D) \leq k$, then there exists a decycling matrix $M$ for $D$ of rank at most $k$. Reversing the arcs of $D$ according to $M$ gives a corresponding  linear order $<$ on the vertices of $D$. The tournament $D$ is recovered from $(<,M)$ by reversing the arcs of the transitive tournament according to $M$. Thus the set of tournaments with $\tmr \leq k$ can be indexed by the set of pairs $(<,M)$ of linear orders and symmetric matrices of rank $ \leq k$. There are $n!$ orders and $\sum_{j \leq k} N(n,j)$ such matrices. The result now follows directly from Lemma \ref{lem:Mc1} with $q = 2$.
 \end{proof}

\section{Proof of Main Results}\label{sec:main}

In this section we prove Theorems \ref{T1}, \ref{T2} and \ref{thm:all_ell_r}. The key result is Theorem  \ref{T2} and the others will follow from the same arguments. Let us first outline the idea of the proof.

It was shown in \cite{natcase} that any decycling matrix of the dijoin $D_1 \to D_2$ has the form
\begin{equation*}
\begin{pmatrix}
    A &C\\ C^{\top} & B
\end{pmatrix}
\end{equation*}
where $A$ and $B$ are decycling matrices for $D_1$ and $D_2$ respectively, and $C$ a staircase matrix recording how the two acyclic orders intertwine. The counterexamples of \cite{Alon,Problems} give  $C=0$, so the ranks of the two diagonal blocks simply add. Our approach is to choose $C$ and $A$ so that the rows of the first block row lie in the row space of the second, such that the rank of the whole matrix is equal to the rank of $B$. Specifically, we take $C=L$ lower triangular, $B$ invertible and $A = L B^{-1}L^{\top}$. 
We first prove Proposition \ref{propcon} which bounds $\tmr(D_1 \rightarrow D_2)$ for $D_2 = T(B)$ and $D_1 = T(LB^{-1}L^{\top})$ arising from such a construction.

The remaining work is then to show that $B$ can be chosen so that both $D_1$ and $D_2$ have `large' tournament minimum rank sufficiently close to their number of vertices.  For $n$ large, a random tournament $D$ on $n$ vertices has $\inv(D) \ge n - 2\sqrt{n\log{n}}$, by a result of \cite[Theorem 1.5]{Problems}. This then suggests that selecting $B$ randomly is likely to work, and this is indeed what we do.

\bigskip

Let us first prove the following proposition, which we will use to get an upper bound on $\tmr(D_1 \rightarrow D_2)$ in Theorem~\ref{T2}.

\begin{prop} \label{propcon}
    Let $n \geq 1$, $B \in S_n$, and set $A \coloneqq LB^{-1}L^{\top}$, where $L$ is the $n \times n$ lower triangular matrix with all entries on and below the diagonal equal to 1. Then $A \in S_n$ and the tournaments
        $D_1 \coloneqq T(A)$ and $D_2 \coloneqq T(B)$ satisfy $\tmr(D_1 \to D_2) \le n$.    
\end{prop}
\begin{proof}
    The matrix $A =  LB^{-1}L^{\top}$ is invertible as it is a product of invertible matrices, and it is symmetric since the inverse of a symmetric matrix is again symmetric. 

     Let $u_1,\ldots,u_n$ be the vertices of $D_1 = T(A)$ and let  $v_1,\ldots,v_n$ be the vertices of $D_2 = T(B)$. 
      By definition, reversing the arcs of $D_1 = T(A)$ as according to $A$ gives the transitive tournament $u_1 <_1 \dots <_1 u_n$ and doing the same for $D_2$ using $B$ gives $v_1 <_2\dots<_2 v_n$.

    Consider the symmetric $2n \times 2n$ matrix
    \begin{equation*}
        M = \begin{pmatrix}
            A &L\\L^{\top} &B
        \end{pmatrix}.
    \end{equation*}
    We will show that $M$ is a decycling matrix for $D_1 \to D_2$, with the rows and columns of M indexed by the vertices in the order $u_1,\ldots,u_n,v_1,\ldots,v_n$. 
    Flip all arcs of $D_1 \to D_2$ according to the places where $M_{ij} = 1$. Following the reversal, arcs within $\{u_1,\ldots,u_n\}$ and $\{v_1,\ldots,v_n\}$ agree with $<_1$ and $<_2$, respectively and hence with $<$. Every arc between $u_i$ and $v_j$ for some $i,j \in \{1,\ldots, n\}$ begins as $u_i \to v_j$ and is reversed exactly when $v_j < u_i$. 
    Hence the arcs of the resulting tournament agree with the alternating order
    \begin{equation*}
        v_1<u_1<v_2<u_2<\dots<v_n<u_n,
    \end{equation*}
   and thus $M$ is decycling.

     Setting 
     \begin{equation*}
         P = \begin{pmatrix}
             LB^{-1} \\I_n
         \end{pmatrix},
     \end{equation*}
    allows us to write
    \begin{equation*}
      PBP^{\top} = \begin{pmatrix}
          LB^{-1}L^{\top} & L \\ L^{\top} & B
      \end{pmatrix}  = M.
    \end{equation*}
    Since $P$ has $n$ columns, $\rank(M) \le n$. Also, as $B$ is a principal submatrix of $M$, $\rank(M) \geq \rank(B) =n$. Hence $\rank(M) =n$ and thus $\tmr(D_1 \to D_2) \le n$.
\end{proof}

We are now in a position to prove Theorems \ref{T1}, \ref{T2} and \ref{thm:all_ell_r}. For the sake of simplicity we will first prove Theorem \ref{T2} and then deduce Theorem \ref{T1} using $\tmr(D) \leq \inv(D) \leq \tmr(D)+1$ as in \eqref{rel}.
\medskip
\begin{proof}[Proof of Theorem \ref{T2}]
    To begin, let $n \in \mathbb{N}$ be large and let 
      \begin{equation*}
        t=  \left \lceil \sqrt{2n\log_2 n+2n}~ \right \rceil.
      \end{equation*}
    Let $k = n-t$, 
    noting that $ 0 \leq k \le n$ for $n$ sufficiently large. 
    
    For convenience, let $\Phi : S_n \to S_n$ be the map sending $\Phi(G) = LG^{-1}L^{\top}$, as used in Proposition \ref{propcon}. Note that $\Phi$ is a bijection with inverse $\Psi(H) = L^{\top} H^{-1} L$.
    We will show there exists some $B \in S_n$ with 
    \begin{equation*}
        \tmr(T(B)) \geq k+1 \text{ and} \tmr(T(\Phi(B))) \geq k+1.
    \end{equation*}
    Proposition \ref{propcon} will then finish the proof. 
  
    Let $T$ be the set of tournaments on $[n]$ with $\tmr \leq k$, so $|T| = M(n,k)$. Set
    \begin{equation*}
        \mathcal{S}_{n, k} = \left\{ G \in S_n : T(G) \in T \right\}
    \end{equation*}
    and note that 
    \[
     \Phi^{-1}(\mathcal{S}_{n, k}) = \left\{ G \in S_n : T(\Phi(G)) \in T \right\}.
    \]
    Any $B \in S_n \setminus ({\mathcal{S}_{n, k} \cup \Phi^{-1}(\mathcal{S}_{n, k})} )$ has the required property so it suffices to show that $ |{{\mathcal{S}_{n, k} \cup \Phi^{-1}(\mathcal{S}_{n, k})}}|< |S_n|$. 
    
    We will start by bounding $ |{{\mathcal{S}_{n, k} \cup \Phi^{-1}(\mathcal{S}_{n, k})}}|$ from above. Since $\Phi$ is a bijection, we have $|\Phi^{-1}(\mathcal{S}_{n, k})| = |\mathcal{S}_{n, k} |$. 
    Hence,
    \begin{align*}
        |{\mathcal{S}_{n, k} \cup \Phi^{-1}(\mathcal{S}_{n, k})} | &\leq 2|\mathcal{S}_{n,k}| 
        \\ &\le  2\cdot 2^n \cdot M(n,k) & \text{by Observation \ref{G1} \eqref{item:2^n diagonal choices}}
        \\ &\le 2^{n+1} n!  \sum_{j=0}^{k} \prod_{i=1}^{\left \lfloor{j/2}\right \rfloor } \frac{2^{2i}}{2^{2i}-1} \times \prod_{i=0}^{j-1} (2^{n-i}-1) & \text{by Corollary \ref{lem:count}}
    \end{align*}
    We bound each factor individually. First,
    \begin{equation*}
        \prod_{i=1}^{\left \lfloor{j/2}\right \rfloor } \frac{2^{2i}}{2^{2i}-1} = \prod_{i=1}^{\left \lfloor{j/2}\right \rfloor }\frac{1}{1-4^{-i}} \le  \prod_{i=1}^{\infty}\frac{1}{1-4^{-i}} \leq 2.
    \end{equation*}
    Next,
    \begin{equation*}
        \prod_{i=0}^{j-1} (2^{n-i}-1) \leq \prod_{i=0}^{j-1} 2^{n-i} = 2^{jn - \binom{j}{2}}.
    \end{equation*}
    Combining these gives
    \begin{equation}\label{eq:firstboundS_nk}
        |{\mathcal{S}_{n, k} \cup \Phi^{-1}(\mathcal{S}_{n, k})} | \leq 2^{n+1} n! \sum_{j=0}^{k} 2^{1+jn - \binom{j}{2}}  
    \end{equation}
     Now we wish to bound $ \sum_{j=0}^{k} 2^{1+jn - \binom{j}{2}}$.
    Define $f(j) \coloneqq 1+jn- \binom{j}{2}$ and note that $f(j+1) - f(j) = n-j \geq n-(n-t+1) \geq t-1 \geq 1$ for $j \le k \le n-t$. Hence, each term is at least twice as large as the one before it. In particular,
    \begin{equation}\label{eq:boundsum}
        \sum_{j=0}^{k} 2^{1+jn - \binom{j}{2}} \leq 2^{1+kn - \binom{k}{2}}(1+ 1/2 + 1/4 + \dots) \leq 2^{2+kn - \binom{k}{2}}.
    \end{equation}
    For $k = n-t$, we have
    \begin{equation}\label{eq:kn-k2}
        kn-\binom{k}{2} = \binom{n+1}{2} - \binom{t+1}{2}.
    \end{equation}
    Substituting \eqref{eq:boundsum} and \eqref{eq:kn-k2} into \eqref{eq:firstboundS_nk}, we obtain
    \begin{equation*}
         |{\mathcal{S}_{n, k} \cup \Phi^{-1}(\mathcal{S}_{n, k})} | \leq 2^{n+1} n! \cdot 2^{2 + \binom{n+1}{2} - \binom{t+1}{2}} \leq  2^{n+3 +n\log_2(n) + \binom{n+1}{2}-\binom{t+1}{2}}.
    \end{equation*}
   
   To bound $|S_n|$, we can apply Lemma \ref{lem:Mc2} together with the inequality $\prod_{i=1}^n (1-x_i) \geq 1 - \sum_{i=1}^n x_i$ for $x_i \in [0,1]$ to obtain
    \begin{equation*}
         |S_n| = 2^{n(n+1)/2} \prod_{1 \leq i \le n, i \text{ odd  }} \lb1-\frac{1}{2^i}\rb \geq 2^{\binom{n+1}{2}} \lb1 - \sum_{1 \leq i \le n, i \text{ odd  }}\frac{1}{2^i}\rb \geq \frac{1}{3} 2^{\binom{n+1}{2}}.
    \end{equation*}
    
    To complete the proof that $|{\mathcal{S}_{n, k} \cup \Phi^{-1}(\mathcal{S}_{n, k})} | < |S_n|$, it therefore suffices to show
    \begin{equation*}
       2^{n+3 +n\log_2(n) + \binom{n+1}{2}-\binom{t+1}{2}} <  \frac{1}{3} 2^{\binom{n+1}{2}}.
    \end{equation*}
    Taking $\log_2 $ of both sides and canceling $ \binom{n+1}{2}$, we require
    \begin{equation*}
        \binom{t+1}{2} > n\log_2 n + n + 3 + \log_2{3}.
    \end{equation*}
    This holds by choice of $t$, for $n$ sufficiently large.
    
    To conclude, it is therefore possible to choose $B \in S_n \setminus ({\mathcal{S}_{n, k} \cup \Phi^{-1}(\mathcal{S}_{n, k})} )$, given $n$ sufficiently large. Set $A = \Phi(B)$, and let $D_1 = T(A)$ and $D_2 = T(B)$. Then $\tmr(D_1) \geq k+1$ and $\tmr(D_2) \geq k+1$ and by Proposition \ref{propcon}, $\tmr(D_1 \to D_2) \leq n$. Hence 
    \begin{equation*}
        \frac{\tmr(D_1 \to D_2)}{\tmr(D_1) + \tmr( D_2)}  \leq \frac{n}{2(k+1)} = \frac{n}{2(n-t+1)} .
    \end{equation*}
    Since $t = o(n)$, this concludes the proof.
\end{proof}

Theorems \ref{T1} and \ref{thm:all_ell_r} now follow from the above proof of Theorem \ref{T2}.

 \begin{proof}[Proof of Theorem \ref{T1}]
Take $n$, $D_1$ and $D_2$ as in Theorem~\ref{T2}. Then $\inv(D_i) \geq \tmr(D_i) > n-O(\sqrt{n\log{n}})$ for all $i \in \{1,2\}$ and $\inv(D_1 \to D_2) \leq \tmr(D_1 \to D_2) +1 \le n+1$. Hence
\begin{align*}
    \frac{\inv(D_1 \to D_2)}{\inv(D_1) + \inv(D_2)} &<  \frac{n+1}{2\lb n-O(\sqrt{n\log{n}})\rb} 
    = \frac{1}{2} + O\lb\sqrt{\frac{\log{n}}{n}}\rb
\end{align*}
and taking $n$ sufficiently large gives the result.
 \end{proof}

Define $[D_1,D_2,\ldots,D_k]$ to be the oriented graph formed by taking successive dijoins $((\ldots(D_1 \rightarrow D_2) \rightarrow \ldots )\rightarrow D_k)$ (noting that the order the dijoins are applied does not matter) and let $[D^t] \coloneqq [\underbrace{D,D,\ldots,D}_{t}]$.
To prove Theorem~\ref{thm:all_ell_r} we will need the following special case of a result of Wang, Yang and Lu.

\begin{lemma}[{\cite[Corollary 1.10]{Wang}}]\label{lem:C3_dijoins}
For every oriented graph $D$ with $\inv(D) \ge 1$ and all $t,s \ge 0$ such that $t+s\ge 1$,
\begin{equation*}
\inv\lb\left[\big(\overrightarrow{C_3}\big)^t,D,\big(\overrightarrow{C_3}\big)^s\right] \rb = \begin{cases} \inv(D) + t +s-1 &\text{if $\inv(\overrightarrow{C_3}\rightarrow D) = \inv(D)$}\\
\inv(D) + t+s &\text{otherwise.}
\end{cases} 
\end{equation*}
\end{lemma}

Theorem \ref{thm:all_ell_r} now follows as a corollary of Lemma \ref{lem:C3_dijoins} and Theorem \ref{T1}.

\begin{proof}[Proof of Theorem \ref{thm:all_ell_r}]
    Take $D_1$ and $D_2$ such that $\inv(D_1 \rightarrow D_2) \le \inv(D_1) + \inv(D_2) - 3$. Such a $D_1$ and $D_2$ exist by Theorem \ref{T1}. We will show that the conclusion of the theorem holds for any $\ell,r \ge \max\lb\inv(D_1),\inv(D_2)\rb$.
    
    Note that by Lemma \ref{lem:C3_dijoins}, there exists $t \in \{\ell - \inv(D_1), \ell - \inv(D_1)+1\}$ such that 
    \begin{equation*}
\inv\lb\left[\big(\overrightarrow{C_3}\big)^{t},D_1\right] \rb = \ell.
\end{equation*}
Let $D_1' = \left[\big(\overrightarrow{C_3}\big)^{t},D_1\right]$.

Similarly, there exists $s\in \{r - \inv(D_2), r - \inv(D_2)+1\}$ such that 
    \begin{equation*}
\inv\lb\left[D_2,\big(\overrightarrow{C_3}\big)^{s}\right] \rb = r.
\end{equation*}
Let $D_2' = \left[D_2,\big(\overrightarrow{C_3}\big)^{s}\right] $.

We have, by Lemma \ref{lem:C3_dijoins} again,
\begin{align*}
\inv(D_1'\rightarrow D_2') &= \inv\left(\left[\big(\overrightarrow{C_3}\big)^t,D_1\rightarrow D_2,\big(\overrightarrow{C_3}\big)^s\right] 
\right) \\
&\le \inv(D_1\rightarrow D_2) + t + s \\
&\le (\inv(D_1) + \inv(D_2)-3) + (\ell - \inv(D_1)+1) +  (r - \inv(D_2)+1) \\
& = \ell + r - 1 = \inv(D_1') + \inv(D_2') -1. 
\end{align*}
Thus $D_1'$ and $D_2'$ have the required properties.
\end{proof}
 
 \section{Concluding Remarks}\label{sec:conclusion}

 The results that we supply show that additivity can fail badly, but they do not provide explicit examples. The proofs are non-constructive, and there is currently no general method in the literature for constructing explicit tournaments with large inversion numbers. Finding explicit tournaments for which additivity fails badly, or even explicit counterexamples to additivity of tournament minimum rank, remains an interesting and open problem.
 
 We note that setting $n = 53$, $t=27$ and $k =26$ in the proof of Theorem \ref{T2}, shows that there exists a counterexample to the additivity conjecture for $|D_1|=|D_2|=53$. We by no means claim that this is the minimum number of vertices for which the conjecture breaks as we expect that many other constructions can lead to an improvement on this bound. Indeed, even just a more careful analysis of the bounds in our proof may well reduce the number of vertices somewhat.
Similarly, in Theorem \ref{thm:all_ell_r} we may take $k = 33$ (using values of $n=63$ and $t=30$ in the proof of Theorem \ref{T2}), but this may not be optimal.

 \bigskip

 Theorem \ref{T1} raises the related interesting question of whether the inversion number of a dijoin is symmetric, that is, whether $\inv(D_1 \to D_2) = \inv(D_2 \to D_1)$ holds for all tournaments $D_1$ and $D_2$. This was conjectured to be true by Wang, Yang and Lu. 
 
   \begin{conj}[{\cite[Conjecture 1.11]{Wang}}]\label{conj}
    For all tournaments $D_1,D_2$ we have
     \begin{equation*}
         \inv(D_1 \to D_2) = \inv(D_2 \to D_1).
     \end{equation*}
 \end{conj}
 
 One might naturally expect this symmetry but our construction suggests that this is not clear. The rank of a decycling matrix for a dijoin depends on how the transitive order of the two tournaments intertwine and reversing the direction of the dijoin alters this intertwining.
 Contrary to Wang, Yang and Lu, we suspect that for sufficiently large  $D_1,D_2$ an asymmetry can occur 
 (and similarly with $\tmr$ in place of $\inv$).

Let $\hat D$ denote the tournament on $V(D)$ attained by flipping every arc of $D$. 
A set of inversions decycle $D$ if and only if they decycle $\hat D$ so $\inv(D) = \inv(\hat D)$. Moreover, $\hat{(D_2 \to D_1)}= \hat D_1 \to \hat D_2$. Hence 
  \begin{equation*}
      \inv(D_2 \to D_1) = \inv\lb\hat{D_2 \to D_1}\rb =\inv( \hat D_1 \to \hat D_2) 
  \end{equation*}
 so Conjecture \ref{conj} reduces to whether the inversion number of a dijoin is unaltered when both parts are reversed. 
 Note that 
 \begin{equation*}
      \lvv \inv(D_1 \to D_2) - \inv(\hat D_1 \to \hat D_2) \rvv \leq 2 
 \end{equation*}
since we can always invert $V(D_1)$ and $V(D_2)$ separately as our first two inversions.
Therefore, any disproof of Conjecture~\ref{conj} would need much more precise bounds on the inversion numbers involved than those given by the methods in this paper.

\section{Acknowledgments}
We acknowledge that AI tools played a role in the research process, even providing key steps in some of the proofs. We as authors have verified and digested such steps and remain firmly responsible for the content and rigor of this paper. The text of this manuscript was written by the authors.
\printbibliography

\end{document}